\documentclass{article}[11pt]

\usepackage{amsmath}   
\usepackage{amssymb}
\usepackage{amsthm}
\usepackage{accents}
\usepackage{xcolor}
\usepackage{tikz-cd}

\newtheorem{Th}{Theorem}[section] 
\newtheorem{Prop}{Proposition}[section]   
   
\newtheorem{Coro}{Corollary}[section]   
  
\newtheorem{Rem}{Remark}[section]

\newcommand{\R}{\mathbb{R}}
\newcommand{\Z}{\mathbb{Z}}

\newcommand{\Sz}{\mathcal{S}}

\newcommand{\A}{{\mathcal A}}
\newcommand{\aA}{{\accentset{\,\,\,\circ}\A}}

\newcommand{\aW}{{\accentset{\,\circ}W}}

\newcommand{\x}{\langle x\rangle}

\newcommand{\p}{{\rm p}}
\newcommand{\s}{\mathbb{S}}

\newcommand{\id}{\text{\rm id}}

\newcommand{\const}{\text{\rm const}}
\newcommand{\tr}{\text{\rm tr}\,}

\newcommand{\Div }{\mathop{\rm div}\nolimits}

\newcommand{\dt}[1]{\accentset{\mbox{\bfseries .}}{#1}}

\newcommand{\dd}{{\rm d}}

\newcommand{\floor}[1]{\lfloor #1 \rfloor}

\begin{document}

\title{Spatial decay/asymptotics in the Navier-Stokes equation}   
 
\author{Peter Topalov} 

\maketitle

\begin{abstract}  
We discuss the appearance of spatial asymptotic expansions of solutions of the Navier-Stokes equation on $\R^n$.
In particular, we prove that the Navier-Stokes equation is locally well-posed in a class of weighted Sobolev and asymptotic spaces.
The solutions depend analytically on the initial data and time and (generically) develop non-trivial asymptotic 
terms as $|x|\to\infty$. In addition, the solutions have a spatial smoothing property that depends on the order of 
the asymptotic expansion.
\end{abstract}   


\section{Introduction}\label{sec:introduction}
In this paper we study the spatial behavior of solutions of the Navier-Stokes equation on the plane $\R^d$, $d\ge 2$,
as a function of the spatial decay of the initial data.
The Navier-Stokes equation describes the time evolution of the velocity field $u(x,t)$, $x\in\R^d$, $t\ge 0$, 
of a viscose fluid,
\begin{equation}
\left\{
\begin{array}{l}\label{eq:NS}
u_t+u\cdot\nabla u=\nu \Delta u-\nabla\p,\quad\Div u=0,\\
u|_{t=0}=u_0,
\end{array}
\right.
\end{equation}
where $\p(x,t)$ is the pressure and the constant $\nu>0$ is the viscosity parameter. 
Here, $u\cdot\nabla :=\sum_{j=1}^d u_j \frac{\partial}{\partial x_j}$ is the derivative in the direction of $u$, 
$\Div u$ is the divergence of $u$, and $\Delta:=\sum_{j=1}^d\frac{\partial^2}{\partial x_j^2}$ is the Laplace operator. 

In order to study the spatial behavior of the solutions of \eqref{eq:NS}, we consider a class of
weighted Sobolev spaces on $\R^d$ as well as a class of asymptotic spaces whose elements
are functions on $\R^d$ that have asymptotic expansions as $|x|\to\infty$, where $|x|$ denotes the Euclidean norm
of $x\in\R^d$. In order to define these spaces, assume that $1<p<\infty$ and let $C_c^\infty(\R^d)$ be the space of 
smooth functions with compact support on $\R^d$ and $\Z_{\ge 0}$ be the set of the non-negative integers.
For a given regularity exponent $m\in\Z_{\ge 0}$ and a weight $\delta\in\R$ the {\em weighted Sobolev space}
$W_\delta^{m,p}\equiv W_\delta^{m,p}(\R^d)$ is defined as the closure of $C_c^\infty(\R^d)$ in the norm 
\begin{equation}\label{W-norm}
\|f\|_{W_\delta^{m,p}}:=
\sum_{|\alpha|\leq m}\big\|\langle x\rangle^{\delta+|\alpha |}\partial^\alpha f\big\|_{L^p},\quad
\x:=\sqrt{1+|x|^2},
\end{equation}
where $\alpha\in\Z_{\ge 0}^d$, $\alpha=(\alpha_1,...,\alpha_d)$, is a multi-index, 
$\partial^\alpha\equiv\partial_1^{\alpha_1}\cdots\partial_d^{\alpha_d}$, 
and $\partial_k\equiv\partial/\partial x_k$ is the (weak) partial derivative in the direction of $x_k$
(\cite{NW,Triebel,McOwen,Lockhart,Bartnik}). For $\alpha=0$ we set $L^p_\delta:=W^{0,p}_\delta$. 
It follows from \cite[Lemma 1.2 $(d)$]{McOwenTopalov2} (cf. also \cite{Bartnik}, \cite[Appendix \ref{ap:A-spaces}]{SultanTopalov}) that 
if $f\in W^{m,p}_\delta(\R^d)$ with $m>\frac{d}{p}$ then $f\in C^k(\R^d)$ for $0\le k<m-\frac{d}{p}$ and we have that
\begin{subequations}
\begin{equation}\label{eq:W-estimate}
\x^{\delta+\frac{d}{p}+|\alpha|}\big|\partial^\alpha f(x)\big|\le C\,\|f\|_{W^{m,p}_\delta}
\quad\text{\rm for}\quad \forall x\in\R^d
\end{equation}
and 
\begin{equation}\label{eq:W-estimate*}
|x|^{\delta+\frac{d}{p}+|\alpha|}\big|\partial^\alpha f(x)\big|\to 0 \ \text{\rm as}\ |x|\to\infty
\end{equation}
\end{subequations}
for any $|\alpha|<m-\frac{d}{p}$.
In addition, for $m>\frac{d}{p}$ and $\delta_1,\delta_2\in\R$ the pointwise multiplication
\begin{equation}\label{eq:W-product}
W^{m,p}_{\delta_1}\times W^{m,p}_{\delta_2}\to W^{m,p}_{\delta_1+\delta_2+\frac{d}{p}},
\quad (f,g)\mapsto fg,
\end{equation}
is bounded (see, e.g., \cite[Proposition 1.2]{McOwenTopalov2}, \cite[Appendix \ref{ap:A-spaces}]{SultanTopalov}). 
In particular, the weighted Sobolev space $W^{m,p}_\delta$ is a Banach algebra for $m>\frac{d}{p}$ and
$\delta+\frac{d}{p}\ge 0$.
Denote by $\aW^{m,p}_\delta$ the closed subspace in $W^{m,p}_\delta$ that consists of the divergence free vector fields
$u\in W^{m,p}_\delta$ and let $B_{\aW^{m,p}_\delta}(\rho)$ be the centered at zero open ball of radius $\rho>0$
in $\aW^{m,p}_\delta$. We also choose a cut-off function $\chi\in C^\infty(\R)$ such that $\chi(\rho)=0$ for 
$|\rho|\le 1$ and $\chi(\rho)=1$ for $|\rho|\ge 2$ and denote by $\s^{d-1}$ the unit sphere in $\R^d$. 
Finally, we set $r:=|x|$ and $\theta:=\frac{x}{|x|}\in\s^{d-1}$ for $x\ne 0$.
The following well-posedness result extends Theorem 1.2, proven in the case of the Euler equation 
in \cite{McOwenTopalov4}. The Theorem establishes a dichotomy of solutions of 
the Navier-Stokes equation depending on the value of the weight parameter $\delta+\frac{d}{p}\ge 0$.

\begin{Th}\label{th:NS_W-spaces}
Assume that $m>2+\frac{d}{p}$ and $\delta+\frac{d}{p}\ge 0$. Then, we have:
\begin{itemize}
\item[(a)] If $0\le\delta+\frac{d}{p}<d+1$ then for any $\rho>0$ there exists $T>0$ such that
for any divergence free initial data $u_0\in B_{\aW^{m,p}_\delta}(\rho)$,
there exists a unique solution of the Navier-Stokes equation \eqref{eq:NS},
\[
u\in C\big([0,T],\aW^{m,p}_\delta\big)\cap C^1\big([0,T],\aW^{m-2,p}_\delta\big)\,.
\]
\item[(b)] If $d+1\le\delta+\frac{d}{p}$ then for any $\rho>0$ there exists $T>0$ such that
for any divergence free initial data $u_0\in B_{\aW^{m,p}_\delta}(\rho)$ there exists a unique solution of
the Navier-Stokes equation \eqref{eq:NS} of the form
\begin{equation}\label{eq:solution(b)}
u(t,x)=\sum_{d+1\le k\le\delta+\frac{d}{p}; k\in\Z}\chi(r)\,\frac{a_k(t,\theta)}{r^k}+f(t,x)
\end{equation}
where we have that\footnote{Note that in view of \eqref{eq:W-estimate*} we have
$f(t,x)=o\big(1/r^{\delta+\frac{d}{p}}\big)$ as $|x|\to\infty$
for any given $t\in[0,T]$.}
\[
f\in C\big([0,T],W^{m,p}_\delta\big)\cap C^1\big([0,T],W^{m-2,p}_\delta\big)
\]
and
\[
a_k\in C^1\big([0,T],C(\s^{d-1},\R^d)\big),\quad d+1\le k\le\delta+\frac{d}{p},\quad k\in\Z,
\]
where $C(\s^{d-1},\R^d)$ denotes the Banach space of continuous maps  $\s^{d-1}\to\R^d$.
In addition, for any given $t\in[0,T]$ and $k$ integer with $d+1\le k\le\delta+\frac{d}{p}$ the components 
$a_k^1(t),...,a_k^d(t)$ of the function $a_k(t) : \s^{d-1}\to\R^d$ are restrictions to the sphere $\s^{d-1}$
of homogeneous harmonic polynomials on $\R^d$ of degree $k-d+2$.
\end{itemize}
The solution in (a) and (b) depends Lipschitz continuously on the initial data $u_0\in W^{m,p}_\delta$
in the sense that the data-to-solution map is Lipschitz continuous. 
Moreover, the solutions depend analytically on time and the initial data in the sense that the map
\[
(0,T)\times B_{\aW^{m,p}_\delta}(\rho)\to\aW^{m,p}_\delta,\quad(t,u_0)\mapsto u(t;u_0),
\]
in the case (a) and the map
\[
(0,T)\times B_{\aW^{m,p}_\delta}(\rho)\to W^{m,p}_\delta\times \big(C(\s^{d-1},\R^d)\big)^{k_*},
\quad(t,u_0)\mapsto
\Big(f(t;u_0),a_k(t;u_0),d+1\le k\le\delta+\frac{d}{p}\Big),
\]
in the case (b), where $k_*$ is the number of integers in the interval $[d+1,\delta+\frac{d}{p}]$, are analytic.
\end{Th}

Theorem \ref{th:NS_W-spaces} is proven in Section \ref{sec:results_W-spaces}.
The result shows that for $m>2+\frac{d}{p}$ the Navier-Stokes equation is locally well-posed
in the weighted Sobolev space $W^{m,p}_\delta$ for $0\le\delta+\frac{d}{p}<d+1$. 
For weights $\delta+\frac{d}{p}\ge d+1$ the space $W^{m,p}_\delta$ is no longer preserved
by the Navier-Stokes flow and (generically) the solutions \eqref{eq:solution(b)} develop 
(cf. Proposition \ref{prop:asymptotics} below) asymptotic terms as $|x|\to\infty$ of the form 
\begin{equation}\label{eq:asymptotic_terms}
\sum\limits_{d+1\le k\le\delta+\frac{d}{p}; k\in\Z}\chi(r)\,\frac{a_k(t,\theta)}{r^k}
\end{equation}
where the components of the asymptotic functions $a_k(t) : \s^{d-1}\to\R$, $d+1\le k\le\delta+\frac{d}{p}$, $k\in\Z$,
are eigenfunctions of the Laplace-Beltrami operator (corresponding to the restriction
of the Euclidean metric on $\R^d$ to $\s^{d-1}$) with eigenvalue $\lambda_k:=k(k-d+2)$ 
(cf., e.g., \cite[Theorem 22.1]{Shubin}).
Note that for $t\in[0,T]$ the term $f(t,x)$ in \eqref{eq:solution(b)} considered as a function of $x\in\R^d$ 
belongs to $W^{m,p}_\delta$ and hence, by \eqref{eq:W-estimate*}, 
it is of order $o\big(1/r^{\delta+\frac{d}{p}}\big)$ as $|x|\to\infty$.
This confirms that the term $f(t,x)$ in \eqref{eq:solution(b)} is a remainder and that the
solution of the Navier-Stokes equation has a spatial asymptotic expansion as $|x|\to\infty$
with asymptotic terms \eqref{eq:asymptotic_terms}. 
Moreover, it follows from Theorem \ref{th:NS_W-spaces}, the estimate \eqref{eq:W-estimate}, 
and the compactness of the interval $[0,T]$, that for any $u_0\in W^{m,p}_\delta$ there exists 
a constant $C_\delta>0$ such that for any $t\in[0,T]$ and $|x|\ge 2$ we have that 
\[
\Big|\partial^\alpha\Big(u(t,x)-
\sum\limits_{d+1\le k\le N-1; k\in\Z}\frac{a_k(t,\theta)}{|x|^k}\Big)\Big|\le
\frac{C_\delta}{|x|^{N+|\alpha|}},\quad 0\le|\alpha|\le L,
\]
where $L\in\Z_{\ge 0}$ is the largest integer such that $0\le L<m-\frac{d}{p}$ and 
$N:=\floor{\delta+\frac{d}{p}}$ with $\floor{x}$ denoting
the largest integer that is less than or equal to $x\in\R$.
This estimate is locally uniform on the choice of the initial data $u_0\in W^{m,p}_\delta$.
Hence, up to an error of order $O\big(1/|x|^{N+|\alpha|}\big)$ the solution $u(t,x)$ of the Navier-Stokes equation
(together with its derivatives of order $|\alpha|\le l$) is represented by its asymptotic part uniformly in
$t\in[0,T]$ and $|x|\ge 2$ and locally uniformly on the initial data. 
In addition, the solution depends analytically on the time and the initial data as described in 
Theorem \ref{th:NS_W-spaces}.
This analytic dependence is in contrast with the analogous result for the solutions of the Euler equation in 
\cite[Theorem 1.2]{McOwenTopalov5}, where the asymptotic functions
$a_k\in C(\s^{d-1},\R^d)$, $d+1\le k\le\delta+\frac{d}{p}$, $k\in\Z$,
are analytic in the time and the initial data, but the remainder 
$f\in C\big([0,T],W^{m,p}_\delta\big)\cap C^1\big([0,T],W^{m-2,p}_\delta\big)$ is not.
In a separate work, we consider the case of the existence of unbounded solutions of the Navier-Stokes
equation in the weighted Sobolev space $W^{m,p}_\delta$ with $-1/2<\delta+\frac{d}{p}<0$.
Such solutions have growth of order $O(|x|^\beta)$ as $|x|\to\infty$ with weight $0<\beta<1/2$.
Note that the appearance of a leading asymptotic term of order $O(1/r^{d+1})$ in the solutions of the 
Navier-Stokes equation for fast decaying initial data is well known (see, e.g., \cite{DS,BM}). 
However, these works do not discuss the well-posedness of the solutions, their dependence on the time and the initial data,
or the existence and the explicit form of the asymptotic terms. 
We will discuss the related work in more detail at the end of this Introduction.

\medskip

The following Proposition shows that the asymptotic terms \eqref{eq:asymptotic_terms} in the solutions 
of the Navier-Stokes equation for $\delta+\frac{d}{p}\ge d+2$ cannot be avoided. 

\begin{Prop}\label{prop:asymptotics}
Assume that $m>2+\frac{d}{p}$ and $\delta+\frac{d}{p}\ge d+1$. Then there exists an open dense set
$\mathcal{N}$ in $\aW^{m,p}_\delta$ such that for any $u_0\in\mathcal{N}$ and for any
$d+1\le k\le\delta+\frac{d}{p}$ and $1\le j\le d$ the $j$-th component of the $k$-th asymptotic
function $a_k(t)\in C(\s^{d-1},\R^d)$ of the solution \eqref{eq:solution(b)} does {\em not} vanish as an element of
$C(\s^{d-1},\R)$ for all but finitely many $t\in[0,T]$.
\end{Prop}

Theorem \ref{th:NS_W-spaces} and Proposition \ref{prop:asymptotics} above show that (generically) the solutions
of the Navier-Stokes equation develop asymptotic terms as $|x|\to\infty$.
These (generically) non-vanishing asymptotic terms are unavoidable obstructions for the invariance of the weighted 
Sobolev space $W^{m,p}_\delta$ for weights $\delta+\frac{d}{p}\ge d+1$. This observation suggest the introduction
of a larger function space that contains functions of the form \eqref{eq:solution(b)} and that is preserved by
the solutions of the equation. Such asymptotic spaces were introduced in \cite[Appendix B]{McOwenTopalov2} 
(cf. also \cite[Appendix C]{SultanTopalov}).
Note that Theorem \ref{th:NS_W-spaces} follows from the results in \cite{McOwenTopalov5} on the solutions
of the Navier-Stokes equation in asymptotic spaces and Proposition \ref{prop:preservation_of_weight_W} proven 
in Section \ref{sec:conservation_law}. 
In order to formulate these results, we proceed to the definition of the {\em asymptotic spaces}. 
To this end, we choose $m>\frac{d}{p}$, $N\in\Z_{\ge 0}$, and set
\begin{equation}\label{def:gamma_N}
\gamma_N:=N+\gamma_0\ \hbox{where $\gamma_0$ is fixed and chosen so that}\ 0\le\gamma_0+\frac{d}{p}<1\, .
\end{equation}
As above, let $\chi\in C^\infty(\R)$ be a cut-off function such that $\chi(\rho)=0$ for 
$|\rho|\le 1$ and $\chi(\rho)=1$ for $|\rho|\ge 2$. 
For integer $0\le n\le N$ and $\ell\geq -n$ define 
$\A^{m,p}_{n,N;\ell}\equiv\A^{m,p}_{n,N;\ell}(\R^d)$ to be the space of real-valued functions of the form 
\begin{equation}\label{AWlog-expansion1} 
v(x)=a(x)+f(x),
\end{equation}
where
\begin{equation}\label{AWlog-expansion2} 
a(x):=\chi(r)\left(\frac{a^0_n(\theta)+\cdots a_n^{n+\ell}(\theta) (\log r)^{n+\ell}}{r^n}+\cdots + 
\frac{a^0_{N}(\theta)+\cdots+a_N^{N+\ell}(\theta) (\log r)^{N+\ell}}{r^{N}}\right),
\end{equation}
$a^j_k\in H^{m+1+N-k,p}(\s^{d-1})$ for $n\le k\le N$, $0\le j\le k+\ell$, and 
\begin{equation}\label{AW-expansion3} 
f\in W^{m,p}_{\gamma_N}\,.
\end{equation}
The function \eqref{AWlog-expansion2} in the decomposition \eqref{AWlog-expansion1} is the
{\em asymptotic part of $v$} and \eqref{AW-expansion3} is the {\em remainder of $v$}.
Since $f\in W^{m,p}_{\gamma_N}$ and $m>\frac{d}{p}$, it follows from
\eqref{eq:W-estimate*} that $f=o\big(1/r^N\big)$ as $|x|\to\infty$. 
The function space $\A^{m,p}_{n,N;\ell}$ is a Banach space
under the norm
\begin{equation}\label{def:AW;-norm}
\|v\|_{\A_{n,N;\ell}^{m,p}}:=\sum_{k=n}^N \sum_{j=0}^{k+\ell} \|a^j_k\|_{H^{m+1+N-k,p}}+
\|f\|_{W_{\gamma_N}^{m,p}}.
\end{equation}
For $n=0$ we write $\A^{m,p}_{N;\ell}$ instead of $\A^{m,p}_{0,N;\ell}$. We are particularly interested in
the case $\ell=0$, which was used in our results on the Euler and the Navier-Stokes equations
(see \cite{McOwenTopalov3,McOwenTopalov5}).
In Appendix \ref{ap:A-spaces} we collect some of the properties of the asymptotic spaces used in this paper.
In particular, note that $\A^{m,p}_{n,N;\ell}$ is a Banach algebra with respect to the pointwise multiplication 
of functions for $-n\le\ell\le 0$  (see Proposition \ref{prop:A-products}(ii)).
For simplicity of notation, we will use (whenever possible) the symbols.
$\A^{m,p}_{n,N;\ell}$, $W^{m,p}_\delta$, etc., independently of whether we are considering single-valued functions, 
vector fields, or tensor fields whose elements belong to the considered space of functions.
We will also need the space of divergence free vector fields on $\R^d$,
$\aA_{n,N;\ell}^{m,p}:=\big\{v\in\A_{n,N;\ell}^{m,p}\,\big|\,\Div v=0\big\}$.
Note that $W^{m,p}_{\gamma_N}$ is the remainder space in $A_{n,N;\ell}^{m,p}$ and hence
$W^{m,p}_{\gamma_N}\subseteq A_{n,N;\ell}^{m,p}$.
The following Proposition is proven in Section \ref{sec:conservation_law}. The Proposition follows from the analog
of the vorticity conservation for the Navier-Stokes equation formulated in Proposition \ref{prop:conservation_law} 
and from Corollary \ref{coro:preservation_of_weight_A} on the preservation of the leading power in 
the asymptotic expansion of the vorticity $\omega$ of $u$.


\begin{Prop}\label{prop:preservation_of_weight_W}
Assume that $m>4+\frac{d}{p}$ and let 
$u\in C\big([0,T],\aA^{m,p}_{N;0}\big)\cap C^1\big([0,T],\aA^{m-2,p}_{N;0}\big)$, $T>0$, 
be a solution of the Navier-Stokes equation \eqref{eq:NS} such that the vorticity at time zero
$\omega_0$ belongs to $W^{m-1,p}_{\delta+1}$ for some $0\le\delta+\frac{d}{p}\le\gamma_N+\frac{d}{p}$. 
Then, $\omega\in C\big([0,T],W^{m-1,p}_{\delta+1}\big)\cap  C^1\big([0,T],W^{m-3,p}_{\delta+1}\big)$.
\end{Prop}

\noindent As mentioned above, Theorem \ref{th:NS_W-spaces} follows from Proposition \ref{prop:preservation_of_weight_W} 
and \cite[Theorem 1.1]{McOwenTopalov5} on the existence of solutions of the Navier-Stokes equation in 
asymptotic spaces (cf. Section \ref{sec:results_W-spaces} for the proof).

In Section \ref{sec:spatial_smoothing} we discuss a {\em spatial smoothing} property of the solutions
of the Navier-Stokes equations. In particular, we prove the following Proposition.

\begin{Prop}\label{prop:smoothingW}
Assume that $m>2+\frac{d}{p}$ and let $0\le\delta+\frac{d}{p}<d+1$.
Take $u_0\in\aW^{m,p}_\delta$ and let
$u\in C\big([0,T_\infty),\aW^{m,p}_\delta\big)\cap C^1\big([0,T_\infty),\aW^{m-2,p}_\delta\big)$
be the solution of the Navier-Stokes equation with initial data $u_0\in\aW^{m,p}_\delta$ given
by Theorem \ref{th:NS_W-spaces} on its maximal interval of existence $T_\infty>0$.
Then $u\in C^\infty\big((0,T_\infty),\aW^{m+j,p}_{\delta-2j}\big)$
for any $\delta-2j+\frac{d}{p}\ge 0$.
\end{Prop}

\noindent In fact, the curve $u : (0,T_\infty)\to\aW^{m+j,p}_{\delta-2j}$ is analytic.


\medskip

There are many important works related to the solutions of the
Navier-Stokes equation in various function spaces. We refer the reader to the monographs
\cite{MB,Chemin} and \cite{Kato1} for a complete list of references. Here we will concentrate
our discussion on works related to the spatial behavior of solutions on $\R^d$ as $|x|\to\infty$.
The fact that the solutions of the Euler and the Navier-Stokes equations on $\R^3$ with
initial data in the Schwartz space have decay of order $O\big(1/|x|^4\big)$ as $|x|\to\infty$
for $t>0$ small was discussed in \cite{DS}. The existence of asymptotic expansions
of solutions of the Euler equation as $|x|\to\infty$ for initial data in the weighted 
Sobolev space $W^{m,p}_\delta(\R^2)$, $\delta+\frac{2}{p}\ge 0$, $m>3+\frac{2}{p}$,
was proven in \cite[Corollary 1.2]{SultanTopalov} (cf. also \cite[Theorem 1.2]{McOwenTopalov4}).
We also mention the work \cite{Bran1}, where the existence of asymptotic expansions of solutions
was proven in the case of the 2d Euler equation for initial data with vorticity with compact support in $L^\infty_c(\R^2)$
as well as for the weak solutions of the 2d Navier-Stokes equations.
The existence of asymptotic expansions (with asymptotic terms expressed in terms of the inverse Fourier transform
of certain functions) of smooth solutions of the Navier-Stokes equation and initial data in the Schwartz space
was obtained in \cite{KR}. In general, these works do not discuss the well-posedness of solutions and their dependence
on the time and initial data.
The well-posedness of classical solutions for the Navier-Stokes equation in asymptotic spaces on $\R^d$ was proven in 
\cite[Theorem 1.1]{McOwenTopalov5}. It was shown in \cite{McOwenTopalov5} that the solutions
depend analytically on time and the initial data. 
Similar results were proven in \cite{McOwenTopalov3,McOwenTopalov4,Cantor1} for the Euler equation. 
As mentioned above, the asymptotic spaces were introduced in \cite{McOwenTopalov2} for studying
the spatial asymptotic behavior of solutions of (non-local) PDE's on $\R^d$ as $|x|\to\infty$.
The present work extends the results proven in \cite{McOwenTopalov4} in the case of the Euler equation to
the Navier-Stokes equation.
In comparison to the previous works, we obtain (local) well-possedness of classical solutions
of the Navier-Stokes equation for initial data $u_0$ in the weighted Sobolev space $W^{m,p}_\delta(\R^d)$, 
$d\ge 2$, for any weigh $\delta+\frac{d}{p}\ge 0$.
In addition, we identify the exact form of the asymptotic terms and prove that the remainder
belongs to the same weighted space $W^{m,p}_\delta$ as the initial data $u_0$. 
The solutions are analytic in time and the initial data and have a spatial smoothing property.
To the best of our knowledge, these results are new.
Finally, note that asymptotic-type spaces with a single asymptotic term of order $O(1/r)$ were introduced
and studied in dimension two: we refer to the affine space $E_m$ in \cite[Definition 1.3.3]{Chemin},
as well as to the radial-energy decomposition in \cite[Definition 3.1]{MB}.

\section{Consequences from the conserved quantity}\label{sec:conservation_law}
In this Section we derive a conserved quantity for the solutions of the Navier-Stokes equation.
The quantity is a direct analog of the vorticity conservation in the case of the Euler equation and is
used in the proof of Proposition \ref{prop:preservation_of_weight_W} at the end of the Section.

Assume that $m>4+\frac{d}{p}$ and let for some $T>0$,
\begin{equation}\label{eq:u-solution}
u\in C\big([0,T],\aA^{m,p}_{N;0}\big)\cap C^1\big([0,T],\aA^{m-2,p}_{N;0}\big),
\end{equation}
be a solution of the Navier-Stokes equation \eqref{eq:NS}.
Denote by $\A D^{m,p}_{N;0}$ the {\em group of asymptotic diffeomorphisms} on $\R^d$,
\[
\A D^{m,p}_{N;0}:=\big\{\varphi(x)=x+f(x), f : \R^d\to\R^d\,\big|\,f\in\A^{m,p}_{N;0}\,\text{ and }
\det(\dd\varphi)>0\big\}
\]
where $\dd\varphi$ denotes the Jacobi matrix of the map $\varphi : \R^d\to\R^d$.
The elements of $\A D^{m,p}_{N;0}$ are diffeomorphisms on $\R^d$ that form a topological group
with respect to the composition of maps (see Theorem \ref{th:A-composition} and 
Proposition \ref{prop:A-composition} in Appendix \ref{ap:A-spaces}).
The map $f\in\A^{m,p}_{N;0}$ appearing in the definition of $\A D^{m,p}_{N;0}$ identifies
$\A D^{m,p}_{N;0}$ with an open set in $\A^{m,p}_{N;0}$ and will be considered
as a coordinate on $\A D^{m,p}_{N;0}$.
For more details on $\A D^{m,p}_{N;0}$ we refer the reader to \cite[Appendix B]{McOwenTopalov2} and
\cite[Section 2]{McOwenTopalov3}. Let us now consider the ordinary differential equation
\begin{equation}\label{eq:ode}
\dt\varphi(t)=u(t)\circ\varphi(t),\quad\varphi|_{t=0}=\id,
\end{equation}
where the dot denotes the differentiation with respect to $t$ in $\A D^{m,p}_{N;0}$ and $\id$ is 
the identity map on $\R^d$.
It follows from Proposition 2.1 in \cite{McOwenTopalov3} that equation \eqref{eq:ode} has a unique solution 
\begin{equation}\label{eq:phi-regularity}
\varphi\in C^1\big([0,T],\A D^{m,p}_{N;0}\big)\,.
\end{equation}
Consider the vorticity 2-form $\omega:=\dd(u^\flat)$ where $\dd$ denotes the exterior differential
acting on differential forms and $u^\flat:=\sum_{j=1}^du_jdx_j$ is the 1-form obtained from the vector field $u$ by 
lowering the indexes with the help of the Euclidean metric on $\R^d$. 
In coordinates,
\begin{equation}\label{eq:omega_kl}
\omega_{kl}:=\frac{\partial u_l}{\partial x_k}-\frac{\partial u_k}{\partial x_l},\quad 1\le k,l\le d\,.
\end{equation}
One easily sees from \eqref{eq:NS} that the vorticity 2-form satisfies the equation
\begin{equation}\label{eq:NS_vorticity_form}
\omega_t+L_u\omega=\nu\Delta\omega
\end{equation}
where $L_u\omega$ denotes the Lie derivative of $\omega$ in the direction of $u$
and the Laplace operator is applied component wise. It follows from \eqref{eq:u-solution} and 
Proposition \ref{prop:A-properties} in Appendix \ref{ap:A-spaces} that
\begin{equation}\label{eq:omega-regularity}
\omega\in C\big([0,T],\A^{m-1,p}_{1,N+1;-1}\big)\cap C^1\big([0,T],\A^{m-3,p}_{1,N+1;-1}\big)\,.
\end{equation}
For $t\in[0,T]$ consider the pull-back $(\varphi^*\omega)(t):=\varphi(t)^*\omega(t)$ of the 2-form $\omega(t)$
with respect to the asymptotic diffeomorphism $\varphi(t) : \R^d\to\R^d$.
In coordinates,
\begin{equation}\label{eq:pull-back_formula}
(\varphi^*\omega)_{\alpha\beta}(t)=\sum_{k,l=1}^d\omega_{kl}(t)\circ\varphi(t)\,
\frac{\partial\varphi_k(t)}{\partial x_\alpha}\frac{\partial\varphi_l(t)}{\partial x_\beta},
\quad 1\le \alpha,\beta\le d,
\end{equation}
with $\varphi(t)=\id+f(t)$, $f\in C^1\big([0,T],\A^{m,p}_{N;0}\big)$.
It now follows from \eqref{eq:pull-back_formula}, \eqref{eq:phi-regularity}, Proposition \ref{prop:A-products}
and Proposition \ref{prop:A-composition} in Appendix \ref{ap:A-spaces}, that
\begin{equation}\label{eq:pull-back_omega-regularity}
\varphi^*\omega\in C\big([0,T],\A^{m-1,p}_{1,N+1;-1}\big)\cap C^1\big([0,T],\A^{m-3,p}_{1,N+1;-1}\big)\,.
\end{equation}
In fact, consider, for example, the term $\omega_{kl}(t)\circ\varphi(t)$, for given $1\le k,l\le d$, appearing in 
\eqref{eq:pull-back_formula} and let us prove that the map
\begin{equation}\label{eq:term*}
[0,T]\to\A^{m-3,p}_{1,N+1;-1},\quad[0,T]\mapsto\omega_{kl}(t)\circ\varphi(t), 
\end{equation} 
is $C^1$. To this end, consider the map
\begin{equation}\label{eq:G}
G : [0,T]\times[0,T]\to\A^{m-3,p}_{1,N+1;-1},\quad(t_1,t_2)\mapsto\omega_{kl}(t_1)\circ\varphi(t_2).
\end{equation}
Since by \eqref{eq:omega-regularity} the map $\omega_{kl} : [0,T]\to\A^{m-3,p}_{1,N+1;-1}$
is $C^1$, we conclude from Proposition \ref{prop:A-composition} in Appendix \ref{ap:A-spaces} that
\begin{equation}\label{eq:G1}
\frac{\partial G}{\partial t_1}(t_1,t_2)=\frac{\dd\omega_{kl}(t_1)}{\dd t_1}\circ\varphi(t_2)\,.
\end{equation}
Similarly, we have from \eqref{eq:omega-regularity} that 
$\omega_{kl}\in C\big([0,T],\A^{m-1,p}_{1,N+1;-1}\big)$, and hence
by Proposition \ref{prop:A-composition} in Appendix \ref{ap:A-spaces}, 
\begin{equation}\label{eq:G2}
\frac{\partial G}{\partial t_2}(t_1,t_2)=
\big(\dd_{\varphi(t_2)}\omega_{kl}(t_1)\big)\Big(\frac{\dd\varphi(t_2)}{\dd t_2}\Big)
\end{equation}
where 
$\dd_{\varphi(t_2)}\omega_{kl}(t_1)\in\A^{m-2,p}_{2,N+2;-2}\subseteq\A^{m-3,p}_{1,N+1;-1}$
denotes the (pointwise) differential of $\omega_{kl}(t_1) : \R^d\to\R$.
The partial derivatives \eqref{eq:G1} and \eqref{eq:G2} when considered as maps
$[0,T]\times[0,T]\to\A^{m-3,p}_{1,N+1;-1}$ are continuous in view of
\eqref{eq:phi-regularity}, Proposition \ref{prop:A-products}
and Proposition \ref{prop:A-composition} in Appendix \ref{ap:A-spaces}.
This implies that the map \eqref{eq:term*} is $C^1$ (cf. \cite[Theorem 3.4.3]{Hamilton}).
The fact that \eqref{eq:pull-back_omega-regularity} holds then follows from 
\eqref{eq:pull-back_formula}, \eqref{eq:phi-regularity}, and Proposition \ref{prop:A-products}
in Appendix \ref{ap:A-spaces}.
For a given $t\in[0,T]$ denote by $\varphi(s;t)$, $t+s\in[0,T]$, the solution of the ordinary differential equation
$\frac{\dd\varphi(s;t)}{\dd s}=u(t+s)\circ\varphi(s;t)$, $\varphi(0;t)=\id$.
By the uniqueness of solutions,
\[
\varphi(t+s)=\varphi(s;t)\circ\varphi(t)
\]
for $t+s\in[0,T]$. This together with \eqref{eq:NS_vorticity_form} implies that for any $t\in[0,T]$,
\begin{align*}
\frac{\dd}{\dd t}(\varphi^*\omega)(t)&=
\frac{\dd}{\dd s}\Big|_{s=0}\varphi(t)^*\big(\varphi(s;t)^*\omega(t+s)\big)
=\varphi(t)^*\big(\omega_t(t)+L_{u(t)}\omega(t)\big)\\
&=\nu\varphi(t)^*\big(\Delta\omega(t)\big)
\end{align*}
where the equality (and the differentiation) holds in the asymptotic space 
$\A^{m-3,p}_{1,N+1;-1}$ (cf. \eqref{eq:pull-back_omega-regularity}).
By integrating this equality, we then obtain that
\begin{equation}\label{eq:conserved_quantity}
(\varphi^*\omega)(t)=\omega(0)+\nu\int_0^t\varphi(s)^*\big(\Delta\omega(s)\big)\,ds,\quad 0\le t\le T,
\end{equation}
where the equality holds in $\A^{m-3,p}_{1,N+1;-1}$.
It follows from \eqref{eq:phi-regularity}, \eqref{eq:omega-regularity}, Proposition \ref{prop:A-properties}, and
Proposition \ref{prop:A-composition}, that 
\[
\varphi^*\big(\Delta\omega\big)\in C\big([0,T],\A^{m-3,p}_{3,N+3;-3}\big)
\]
and hence, the integral in \eqref{eq:conserved_quantity} exists as a Riemann integral in 
$\A^{m-3,p}_{3,N+3;-3}\subseteq\A^{m-3,p}_{1,N+1;-1}$.

\begin{Prop}\label{prop:conservation_law}
Assume that $m>4+\frac{d}{p}$ and let 
$u\in C\big([0,T],\aA^{m,p}_{N;0}\big)\cap C^1\big([0,T],\aA^{m-2,p}_{N;0}\big)$, $T>0$, 
be a solution of the Navier-Stokes equation \eqref{eq:NS}. Then, \eqref{eq:conserved_quantity} holds 
in $\A^{m-3,p}_{1,N+1;-1}$ and
\[
\varphi(t)^*\omega(t)-\omega(0)=\nu\int_0^t\varphi(s)^*\big(\Delta\omega(s)\big)\,ds
\in\A^{m-3,p}_{3,N+3;-3}
\]
for any $0\le t\le T$ where the integral exists as a Riemann integral.
\end{Prop}

As a consequence from Proposition \ref{prop:conservation_law} we obtain the following Corollary.

\begin{Coro}\label{coro:preservation_of_weight_A}
Assume that $m>4+\frac{d}{p}$ and let 
$u\in C\big([0,T],\aA^{m,p}_{N;0}\big)\cap C^1\big([0,T],\aA^{m-2,p}_{N;0}\big)$, $T>0$, 
be a solution of the Navier-Stokes equation \eqref{eq:NS} such that the vorticity at time zero
$\omega_0$ belongs to $\A^{m-1,p}_{n,N+1;-1}$ for some $1\le n\le N+1$. Then,
$\omega\in C\big([0,T],\A^{m-1,p}_{n,N+1;-1}\big)\cap  C^1\big([0,T],\A^{m-3,p}_{n,N+1;-1}\big)$.
\end{Coro}

\noindent In other words, Corollary  \ref{coro:preservation_of_weight_A} states that the leading power
in the asymptotic expansion of the vorticity is preserved by the equation.

\begin{proof}[Proof of Corollary \ref{coro:preservation_of_weight_A}]
Assume that $\omega_0\in\A^{m-1,p}_{n,N+1;-1}$ for some $1\le n\le N$.
Then we obtain from Proposition \ref{prop:conservation_law} (cf. also \eqref{eq:conserved_quantity}), 
and Theorem \ref{th:A-composition} in Appendix \ref{ap:A-spaces},
that for any $t\in[0,T]$ the leading term in the asymptotic expansion of $\varphi(t)^*\omega(t)$
is of order $O(1/r^{n'})$ where $n':=\min(n,3)$.
This and the relation $\omega(t)=\psi(t)^*[\varphi(t)^*\omega(t)]$ where
$\psi(t):=\varphi(t)^{-1}\in\A D^{m,p}_{N;0}$ implies that the leading term in the asymptotic expansion of 
$\omega(t)$ is of order $O(1/r^{n'})$ and hence
\begin{equation}\label{eq:eqlt1}
\omega\in C\big([0,T],\A^{m-1,p}_{n',N+1;-1}\big)\cap C^1\big([0,T],\A^{m-3,p}_{n',N+1;-1}\big),
\end{equation}
and hence, the Corollary is proven in the case when $1\le n\le 3$. If $n>3$ we apply
Proposition \ref{prop:conservation_law} and Theorem \ref{th:A-composition} again to conclude that for any $t\in[0,T]$, 
\[
\omega(t)\in\A^{m-1,p}_{n',N+1;-1},\quad n':=\min(n,5),
\]
and that \eqref{eq:eqlt1} holds with $n'=\min(n,5)$.
This proves the Corollary for $1\le n\le 5$. The case of $n>5$ follows by repeating 
these arguments inductively.
\end{proof}

Let us now prove Proposition \ref{prop:preservation_of_weight_W} stated in the Introduction.

\begin{proof}[Proof of Proposition \ref{prop:preservation_of_weight_W}]
Assume that $\omega_0\in W^{m-1,p}_{\delta+1}$ for some $0\le\delta+\frac{d}{p}\le\gamma_N+\frac{d}{p}$ and
set $N':=\floor{\delta+\frac{d}{p}}\le N$ and  $\gamma_{N'}:=\delta$. 
Then $\omega_0\in W^{m-1,p}_{\delta+1}\subseteq\A^{m-1,p}_{N'+1,N'+1;0}$.
Since $u\in\A^{m,p}_{N;0}\subseteq\A^{m,p}_{N';0}$ we see that
$u\in C\big([0,T],\A^{m,p}_{N';0}\big)\cap C\big([0,T],\A^{m-2,p}_{N';0}\big)$ is a solution of \eqref{eq:NS}.
We then obtain from Corollary \ref{coro:preservation_of_weight_A} that
$\omega\in C\big([0,T],\A^{m-1,p}_{N'+1,N'+1;0}\big)\cap C^1\big([0,T],\A^{m-3,p}_{N'+1,N'+1;0}\big)$.
Then, 
\[
\Delta\omega(t)\in\A^{m-3,p}_{N'+3,N'+3;-2}\subseteq\A^{m-3,p}_{N'+2,N'+2;-2}
\] 
for any $t\in[0,T]$. Since $\varphi\in C\big([0,T],\A D^{m,p}_{N';0}\big)$ by \cite[Proposition 2.1]{McOwenTopalov3}, we
conclude from \eqref{eq:pull-back_formula} (with $\omega$ replaced by $\Delta\omega$), Proposition \ref{prop:A-products}, 
Proposition \ref{prop:A-composition} in Appendix \ref{ap:A-spaces}, that the integrand in \eqref{eq:conserved_quantity}
belongs to $C\big([0,T],\A^{m-3,p}_{N'+2,N'+2;-2}\big)$ and hence
$C\big([0,T],W^{m-3,p}_{\delta+1}\big)$ in view of the inclusion 
$\A^{m-3,p}_{N'+2,N'+2;-2}\subseteq W^{m-3,p}_{\delta+1}$. 
This implies that the integral in \eqref{eq:conserved_quantity} belongs to $W^{m-3,p}_{\delta+1}$, and hence
the asymptotic part of $\omega(t)$ in $\A^{m-1,p}_{N'+1,N'+1;0}$ vanishes and
$\omega(t)\in W^{m-1,p}_{\delta+1}$ for any $t\in[0,T]$. This proves that
$\omega\in C\big([0,T],W^{m-1,p}_{\delta+1}\big)\cap  C^1\big([0,T],W^{m-3,p}_{\delta+1}\big)$.
\end{proof}

\section{Proof of Theorem \ref{th:NS_W-spaces} and Proposition \ref{prop:asymptotics}}\label{sec:results_W-spaces}
In this Section we concentrate our attention to the case when the initial data $u_0$ belongs
to the weighted Sobolev space $W^{m,p}_\delta$ with $m>2+\frac{d}{p}$ and $\delta+\frac{d}{p}\ge 0$ and prove
Theorem \ref{th:NS_W-spaces} and Proposition \ref{prop:asymptotics} stated in the Introduction.

\begin{proof}[Proof of Theorem \ref{th:NS_W-spaces}]
Assume that $m>2+\frac{d}{p}$ and $\delta+\frac{d}{p}\ge 0$ and take $u_0\in\aW^{m,p}_\delta$.
Since $\aW^{m,p}_\delta$ is boundedly embedded in the asymptotic space $\aA^{m,p}_{N,N;0}$
with $N:=\floor{\delta+\frac{d}{p}}$, $\gamma_N:=\delta$, it follows from \cite[Theorem 1.1]{McOwenTopalov4} that
for any $\rho>0$ there exists $T>0$ such that for any divergence free initial data
$u_0\in B_{\aW^{m,p}_\delta}(\rho)$ there exists a unique solution of the Navier-Stokes equation \eqref{eq:NS},
\begin{equation}\label{eq:u_in_A}
u\in C\big([0,T],\aA^{m,p}_{N;0}\big)\cap C^1\big([0,T],\aA^{m-2,p}_{N;0}\big)\,.
\end{equation}
This solution depends continuously on the initial data $u_0\in\aW^{m,p}_\delta$ and, by
\cite[Remark 4.5]{McOwenTopalov5}, it is analytic when considered as a map
$(0,T)\to\aA^{m,p}_{N;0}$. In order to indicate the dependence of the solution \eqref{eq:u_in_A}
on the initial data $v\in B_{\aW^{m,p}_\delta}(\rho)$ we will use the notation $u(t;v)$, $t\in[0,T]$,
for the solution, and $\omega(t;v)$ for the corresponding vorticity form.
Then, $\omega_0\equiv\dd u_0^\flat\in W^{m-1,p}_{\delta+1}$ and if $m>4+\frac{d}{p}$ we obtain from
Proposition \ref{prop:preservation_of_weight_W} that
\begin{equation}\label{eq:omega_in_W}
\omega\in C\big([0,T],W^{m-1,p}_{\delta+1}\big)\cap  C^1\big([0,T],W^{m-3,p}_{\delta+1}\big).
\end{equation}
If $2+\frac{d}{p}<m\le 4+\frac{d}{p}$ we choose $\overline{m}>4+\frac{d}{p}$. 
Then, by the discussion above, for any 
$v\in B_{\aW^{m,p}_\delta}(\rho)\cap\aW^{\overline{m},p}_\delta$
there exist $0<T_v\le T$ and a unique solution of the Navier-Stokes equation 
$u(t;v)$, $t\in[0,T]$, such that \eqref{eq:u_in_A} holds and
\begin{equation}\label{eq:omega_in_A,m-bar}
\omega(\cdot;v)\in C\big([0,T_v],W^{\overline{m}-1,p}_{\delta+1}\big)\cap
C^1\big([0,T_v],W^{\overline{m}-3,p}_{\delta+1}\big)\,.
\end{equation}
In view of the analyticity of the solution \eqref{eq:u_in_A} in time $t\in(0,T)$ we conclude from
\eqref{eq:omega_in_A,m-bar} that the asymptotic terms of $\omega(t;v)$ vanish for any $t\in[0,T]$.
This implies that \eqref{eq:omega_in_W} holds for any 
$v\in B_{\aW^{m,p}_\delta}(\rho)\cap\aW^{\overline{m},p}_\delta$.
In view of the continuous dependence of the solution \eqref{eq:u_in_A} on the initial data
$v\in B_{\aW^{m,p}_\delta}(\rho)$ and the density of the embedding 
$W^{\overline{m},p}_\delta\subseteq W^{m,p}_\delta$ we then conclude by approximation that 
\eqref{eq:omega_in_W} holds for any $m>2+\frac{d}{p}$.

Since in coordinates $\omega_{jq}=\frac{\partial u_q}{\partial x_j}-\frac{\partial u_j}{\partial x_q}$,
$1\le j,q\le d$, we obtain that
\[
\sum_{j=1}^d\frac{\partial\omega_{jq}}{\partial x_j}=
\Delta u_q-\frac{\partial(\Div u)}{\partial x_q}=\Delta u_q.
\]
Let us first consider the case when $\delta+\frac{d}{p}\notin\Z_{\ge 0}$.
Then\footnote{This is the Biot-Savart identity.}
\begin{equation}\label{eq:Biot-Savart}
u_q=\Delta^{-1}(\Div\omega)_q,\quad 1\le q\le d,
\end{equation}
where $(\Div\omega)_q:=\sum_{j=1}^d\frac{\partial\omega_{jq}}{\partial x_j}\in W^{m-2,p}_{\delta+2}$ and 
we used the fact that
\[
\Delta^{-1} : W^{m-2,p}_{\delta+2}\to\mathcal{J}^{m,p}_\delta
\]
is an isomorphism of Banach spaces for $\delta+\frac{d}{p}>0$ (see \cite[Proposition B.1]{McOwenTopalov4}). 
The space $\mathcal{J}^{m,p}_\delta$ is a subspace of the asymptotic space $\A^{m,p}_{N;0}$ that
consists of elements of the form
\begin{equation}\label{eq:pre-form}
v=\chi(r)\sum_{d-2\le k<\delta+\frac{d}{p}; k\in\Z}\frac{a_k(\theta)}{r^k}+f,\quad f\in W^{m,p}_\delta,
\end{equation}
where the components of $a_k : \s^{d-1}\to\R^d$ are the restriction to $\s^{d-1}$ of homogeneous harmonic polynomials 
on $\R^d$ of degree $k':=k-d+2$ (see \cite[Proposition B.1]{McOwenTopalov4}). 
Hence, we have that $u(t)\in\mathcal{J}^{m,p}_\delta$ for any $t\in[0,T]$. 
In particular, we see that
\begin{equation}\label{eq:u_in_W}
u\in C\big([0,T],W^{m,p}_\delta\big)\cap  C^1\big([0,T],W^{m-2,p}_\delta\big).
\end{equation}
for $0<\delta+\frac{d}{p}<d-2$. Assume that $d-2<\delta+\frac{d}{p}\notin\Z_{\ge 0}$.
Let $\nu(k')$ be the dimension of the linear space $\mathcal{H}_{k'}$ of homogeneous harmonic polynomials of
degree $k'\in\Z_{\ge 0}$ on $\R^d$. For $k'\in\Z_{\ge 0}$ we choose a basis 
$\big\{H_{k';l}\,\big|\,1\le l\le\nu(k')\big\}$ in $\mathcal{H}_{k'}$. 
Recall that the map that assigns to any homogeneous harmonic polynomial of
degree $k'\ge 0$ on $\R^d$ its restriction to the sphere $\s^{d-1}$ is an isomorphism between the space 
$\mathcal{H}_{k'}$ and the eigenspace of the (positive) Laplace-Beltrami operator on $\s^{d-1}$ with 
eigenvalue $\lambda_{k'}:=k'(k'+d-2)$ (see, e.g., \cite[\S 22.2]{Shubin}). 
Take $k'\in\Z_{\ge 0}$ such that $d-2\le k<\delta+\frac{d}{p}\notin\Z_{\ge 0}$ and let $1\le l\le\nu(k')$.
Then, by \cite[Proposition B.1]{McOwenTopalov4} the Fourier coefficient $\widehat{a}_{k';l}^q$ of 
$\Delta^{-1}(\Div\omega)_q$ (cf. \eqref{eq:Biot-Savart}) that corresponds to the eigenfunction 
$h_{k';l}:=H_{k';l}\big|_{\s^d}\in C(\s^{d-1},\R)$ of the Laplace-Beltrami operator on $\s^{d-1}$ is given by
\begin{equation}\label{eq:the_fourier_coefficient}
\widehat{a}_{k';l}^q=-C_{d,k'}\int_{\R^d}(\Div\omega)_q(x)\,H_{k';l}(x)\,dx
\end{equation}
where $C_{d,k'}>0$ is a constant depending only on the choice of $d$ and $k'$.
Let us now show that the first three asymptotic functions $a_{d-2}(t)$, $a_{d-1}(t)$, and $a_d(t)$
do not appear for any $t\in[0,T]$ in the expansion \eqref{eq:pre-form} of the solution $u(t)$ given by \eqref{eq:Biot-Savart}.
The vanishing of $a_{d-2}(t)$ follows from \eqref{eq:the_fourier_coefficient} (with $k'=0$ and $H_{k';l}(x)=\const$), 
the divergence theorem, and the fact that $\omega(t)=o\big(1/r^{d-1}\big)$ as $|x|\to\infty$ in view of
\eqref{eq:W-estimate*} and the fact that $d-2<\delta+\frac{d}{p}\notin\Z_{\ge 0}$ and $\omega(t)\in W^{m-1,p}_{\delta+1}$.
Hence, \eqref{eq:u_in_W} holds for $d-2<\delta+\frac{d}{p}<d-1$.
In order to see that the other two asymptotic functions also vanish, we assume that $d-1<\delta+\frac{d}{p}\notin\Z_{\ge 0}$ 
and recall from \cite[Section 4]{McOwenTopalov5} that the Navier-Stokes equation can be written in the form
\begin{equation}\label{eq:NSQ}
u_t+u\cdot\nabla u=\nu\Delta u+\Delta^{-1}\big(\nabla Q(u)\big),\quad 
Q(u):=\tr\big((\dd u)^2\big)=\Div\big(u\cdot\nabla u\big),
\end{equation}
where $\dd u$ denotes the Jacobi matrix of $u\in\mathcal{J}^{m,p}_\delta$.  
In view of \eqref{eq:pre-form}, \eqref{eq:W-estimate*},  and the vanishing of $a_{d-2}$,
we have that $u\cdot\nabla u=O\big(1/r^{2d-1}\big)$ and $\Delta u=O\big(1/r^{d+1}\big)$ as $|x|\to\infty$. 
By comparing the leading terms in the asymptotic expansions of the both hand sides of \eqref{eq:NSQ} as $|x|\to\infty$
we then obtain that the leading two terms in the asymptotic expansion of $u_t$ are 
generated by $\Delta^{-1}\big(\nabla Q(u)\big)$. 

\begin{Rem}\label{rem:J_in_W}
It follows from \eqref{eq:pre-form}, \eqref{eq:W-estimate*}, and the vanishing of $a_{d-2}$, 
that for any given $t\in[0,T]$ we have that $u(t)=O\big(1/r^{d-1}\big)$ as $|x|\to\infty$ and that 
$u(t)\in W^{m,p}_{\delta_1}$ for some $d-2<\delta_1+\frac{d}{p}<d-1$ independent on $t$. 
By the same reasoning, we obtain from the second relation in \eqref{eq:NSQ} that 
$\nabla Q(u(t))=O\big(1/r^{2d+1}\big)$ as $|x|\to\infty$ and $\nabla Q(u(t))\in W^{m-2,p}_{\delta_2+2}$
for some $d+2\le 2d<\delta_2+2<2d+1$ independent on $t$.
Since $\delta_2>d$ we can apply \cite[Proposition B.1]{McOwenTopalov4} to conclude 
that $\Delta^{-1}\big(\nabla Q(u(t))\big)\in J^{m,p}_{\delta_2}$ with a remainder function
of order $o\big(1/r^d\big)$ as $|x|\to\infty$.
\end{Rem}

\noindent It now follows from \eqref{eq:NSQ}, Remark \ref{rem:J_in_W}, and \cite[Proposition B.1]{McOwenTopalov4},
that for any $t\in[0,T]$, $k'=1,2$, and for any $1\le q\le d$, $1\le l\le\nu(k')$, we have that
\begin{equation}\label{eq:a-derivative}
\frac{\dd}{\dd t}\,\widehat{a}_{k';l}^q(t)=\widehat{b}_{k';l}^q(t)\quad\text{\rm and}\quad
\widehat{b}_{k';l}^q(t):=-C_{d,k'}\int_{\R^d}H_{k';l}\,\partial_q\big(Q(u(t))\big)\,dx
\end{equation}
where $\widehat{b}_{k';l}^q(t)$ denotes the Fourier coefficient of the $q$-th components of
the $k$-th asymptotic term of $\Delta^{-1}\big(\nabla Q(u)\big)$ corresponding to the eigenfunction $h_{k';l}$.
It follows from the divergence theorem (applied three times) and the second relation in \eqref{eq:NSQ} that
\begin{equation}\label{eq:Fourier_coefficient_relation}
\int_{\R^d}H_{k';l}\,\partial_q\big(Q(v)\big)\,dx=
-\int_{\R^d}\sum_{1\le\alpha,\beta\le d}
\frac{\partial^3 H_{k';l}}{\partial x_\alpha\partial x_\beta\partial x_q}(x)\,v_\alpha(x) v_\beta(x)\,dx=0
\end{equation}
for any divergence free vector field $v\in C^\infty_c(\R^d)$ with compact support
(see, e.g., formula (64) and (66) in \cite{McOwenTopalov4}). 
The right hand side in \eqref{eq:Fourier_coefficient_relation} vanishes since the degree of 
the homogeneous polynomial $H_{k';l}$ is assumed to be one or two.
By the density of the embedding $C^\infty_c\subseteq W^{m,p}_{\delta_1}$ (cf. Remark \ref{rem:J_in_W}) 
and the continuity of the left hand side of \eqref{eq:Fourier_coefficient_relation} as a function of 
$v\in W^{m,p}_{\delta_1}$, we conclude that the left hand side of \eqref{eq:Fourier_coefficient_relation}
vanishes  and for any $k'=1,2$ and for any $v\in J^{m,p}_\delta\subseteq W^{m,p}_{\delta_1}$ with vanishing $a_{d-2}$.
It then follows from \eqref{eq:a-derivative} that $a_{d-1}(t)=a_{d-1}(0)=0$ and  $a_d(t)=a_d(0)=0$
for any $t\in[0,T]$. This proves that (a) and (b) hold in the case when $\delta+\frac{d}{p}\notin\Z_{\ge 0}$.
The case $\delta+\frac{d}{p}\in\Z_{\ge 0}$ then follows from the continuous dependence of the solution
\eqref{eq:u_in_A} on the initial data $u_0\in W^{m,p}_\delta$ and the case when 
$\delta+\frac{d}{p}\notin\Z_{\ge 0}$. In fact, assume that $\kappa:=\delta+\frac{d}{p}\in\Z_{\ge 0}$.
Take $\delta'\in\R$ such that $\delta'+\frac{d}{p}\in(\kappa,\kappa+1)$. Then 
$W^{m,p}_{\delta'}$ is densely embedded in $W^{m,p}_\delta$. 
Take a sequence $u_{0n}\in W^{m,p}_{\delta'}$, $n\ge 1$, such that
\[
u_{0n}\stackrel{W^{m,p}_\delta}{\longrightarrow}u_0,\quad n\to\infty\,.
\]
In particular, the sequence converges in the asymptotic space $\A^{m,p}_{N;0}$ and
by the continuous dependence of the solutions  of the Navier-Stokes equation on the initial data
we obtain that for any $t\in[0,T]$ the asymptotic functions of the solution $u(t;u_{0n})$ converge
uniformly on $\s^{d-1}$ to the asymptotic functions of the solution $u(t;u_0)$
(cf. the definition of the norm \eqref{def:AW;-norm}).
Since the solution $u(t;u_{0n})$ is of the form \eqref{eq:solution(b)}, we obtain that
$u(t;u_0)$ has the same form.
This proves items (a) and (b) in the general case.
The analyticity of the solution follows from \cite[Remark 4.5]{McOwenTopalov5}.
\end{proof}

\begin{proof}[Proof of Proposition \ref{prop:asymptotics}]
The proof of this Proposition is the same as the proof of \cite[Proposition 1.1]{McOwenTopalov4} and
is based on \cite[Lemma 5.1]{McOwenTopalov4}.
We omit the details.
\end{proof}

\section{The spatial smoothing}\label{sec:spatial_smoothing}
In this Section we prove a spatial smoothing property of the solutions of the Navier-Stokes equation in
weighted Sobolev and asymptotic spaces. 

\begin{proof}[Proof of Proposition \ref{prop:smoothingW}]
Assume that $m>2+\frac{d}{p}$ and let $0\le\delta+\frac{d}{p}<d+1$.
Take $u_0\in\aW^{m,p}_\delta$ and let
$u\in C\big([0,T_\infty),\aW^{m,p}_\delta\big)\cap C^1\big([0,T_\infty),\aW^{m-2,p}_\delta\big)$
be the solution of the Navier-Stokes equation with initial data $u_0\in\aW^{m,p}_\delta$ given
by Theorem \ref{th:NS_W-spaces} on its maximal interval of existence $T_\infty>0$.
Then, by \cite[Remark 4.5]{McOwenTopalov5}, 
\[
u\in C^k\big((0,T_\infty),\aW^{m,p}_\delta\big)
\]
for any given $k\ge 1$. In particular, $u$ satisfies the non-homogeneous heat equation
\begin{equation}\label{eq:non-homogeneous_linear_equation}
u_t=\nu\Delta u+f(t),\quad u|_{t=0}=u_0,
\end{equation}
where $f(t):=F(u(t))$, 
\begin{align*}
F(u):=\Delta^{-1}\circ\nabla\circ Q(u)-u\cdot\nabla u,\quad Q(u)\equiv\tr\big((\dd u)^2\big)
\end{align*}
(see \cite[Section 4]{McOwenTopalov5}).
It follows from \cite[Proposition B.1]{McOwenTopalov4} and formula \eqref{eq:Fourier_coefficient_relation} 
(see the arguments in the proof of Theorem \ref{th:NS_W-spaces}) that the map
\[
F_1 : W^{m,p}_\delta\to W^{m,p}_{\delta+1},\quad u\mapsto\Delta^{-1}\big(\nabla Q(u)\big),
\]
is well-defined and analytic. Similarly, in view of \eqref{eq:W-product} the map
and $F_2 : W^{m,p}_\delta\to W^{m-1,p}_{\delta+1}$, $u\mapsto u\cdot\nabla u$,
is analytic. This implies that the map $F : W^{m,p}_\delta\to W^{m-1,p}_\delta$ is analytic, and hence
\[
f\in C\big([0,T_\infty),W^{m-1,p}_\delta\big)\cap C^1\big((0,T_\infty),W^{m-1,p}_\delta\big)
\]
for any $k\ge 0$.
Since the Laplace operator $\Delta$ when considered as an unbounded operator on $W^{m-1,p}_\delta$
with domain 
$\widetilde{W}^{m+1,p}_\delta:=\big\{v\in\Sz'\,\big|\, \partial^\alpha v\in W^{m-1,p}_\delta, |\alpha|\le 2\big\}$
generates an analytic semigroup (\cite[Theorem 2.2]{McOwenTopalov5}), we conclude from
\cite[Theorem 3.2.2]{Henry} and the uniqueness of solutions of the heat equation 
(cf., e.g., \cite[Proposition B.1]{McOwenTopalov5}) that
\begin{equation}\label{eq:eqlt3}
u\in C\big((0,T_\infty),\widetilde{W}^{m+1,p}_\delta\big)\cap C^1\big((0,T_\infty),W^{m-1,p}_\delta\big)\,.
\end{equation}
Recall from \cite[Lemma B.7]{McOwenTopalov5} that
$\widetilde{W}^{m+1,p}_\delta=W^{m-1,p}_\delta\cap W^{m,p}_{\delta-1}\cap W^{m+1,p}_{\delta-2}$.
By combining this with \eqref{eq:eqlt3} we conclude that $u\in C\big((0,T_\infty),\aW^{m+1,p}_{\delta-2}\big)$.
Assume that $\delta-2+\frac{d}{p}\ge 0$. Then, we take an arbitrary $t_0\in(0,T_\infty)$ and consider 
the Navier-Stokes equation with initial data $v|_{t=t_0}=u(t_0)\in\aW^{m+1,p}_{\delta-2}$. 
By Theorem \ref{th:NS_W-spaces} there exists $\epsilon>0$, $t_0+\epsilon<T$, and a solution 
$v\in C\big([t_0,t_0+\epsilon),\aW^{m+1,p}_{\delta-2}\big)\cap 
C\big([t_0,t_0+\epsilon),\aW^{m-1,p}_{\delta-2}\big)$
of the Navier-Stokes equation with initial data $u(t_0)$ at $t=t_0$. 
By \cite[Remark 4.5]{McOwenTopalov5}, $v\in C^k\big((t_0,t_0+\epsilon),\aW^{m+1,p}_{\delta-2}\big)$
for any $k\ge 0$. In view of the uniquenss of the solutions of the Navier-Stokes equation in the class 
$C\big([t_0,t_0+\epsilon),\aW^{m,p}_{\delta-2}\big)\cap 
C\big([t_0,t_0+\epsilon),\aW^{m-2,p}_{\delta-2}\big)$ (apply again Theorem \ref{th:NS_W-spaces})
we then conclude that $v=u_{[t_0,t_0+\epsilon)}$. This implies that 
\[  
u\in C^k\big((0,T_\infty),\aW^{m+1,p}_{\delta-2}\big)
\]
for any $k\ge 1$. In particular, 
$u\in C\big((0,T_\infty),\aW^{m+1,p}_{\delta-2}\big)\cap C^1\big((0,T_\infty),\aW^{m-1,p}_{\delta-2}\big)$.
This allows us to repeat the above argument $j_*$ times (where $j_*\ge 0$ is
the maximal integer such that $\delta-2j_*+\frac{d}{p}\ge 0$)
and to conclude the statement of the Proposition inductively.
\end{proof}



\appendix

\section{Basic properties of asymptotic spaces}\label{ap:A-spaces}
In this Appendix we summarize for the convenience of the reader some of the basic properties of
the weighted and the asymptotic spaces introduced in the Introduction. 
For further details we refer to \cite[Appendix B]{McOwenTopalov2} and \cite[Appendix C]{McOwenTopalov4}. 
Assume that $1<p<\infty$ and recall from \cite[Lemma 2.2]{McOwenTopalov2} that for any $m\ge 0$, 
$\delta\in\R$, $1\le k\le d$, $\partial_k : W^{m+1,p}_\delta\to W^{m,p}_{\delta+1}$ is bounded. 

\begin{Prop}\label{prop:A-properties}  
Assume that $m\ge 0$. Then, we have
\begin{enumerate}
\item[(i)] If $n_1\geq n$, $N_1\geq N$, and $\ell_1\leq\ell$, then we have a bounded inclusion 
${\mathcal A}_{n_1,N_1;\ell_1}^{m,p}\subseteq {\mathcal A}_{n,N;\ell}^{m,p}$.
\item[(ii)]  If $m\geq 1$,  then $u\mapsto \partial u/\partial x_j$ is a bounded linear map 
${\mathcal A}_{n,N;\ell}^{m,p}\to{\mathcal A}_{n+1,N+1;\ell-1}^{m-1,p}$.
\item[(iii)] Multiplication by $\chi(r)\,r^{-k}$ is bounded ${\mathcal A}_{n,N;\ell}^{m,p}\!\to\! 
{\mathcal A}_{n+k,N+k;\ell-k}^{m,p}$.
\item[(iv)] Multiplication by $\chi(r)\,(\log r)^{j}$ is bounded 
${\mathcal A}_{n,N;\ell}^{m,p}\!\to\! {\mathcal A}_{n,N^-;\ell+j}^{m,p}$
for any $N^-<N$.
\end{enumerate}
\end{Prop}

\noindent We also consider products of functions in asymptotic spaces. 
Recall from the Introduction (cf. \cite[Proposition 1.2, Lemma 2.2]{McOwenTopalov2},
\cite[Appemdix A]{SultanTopalov}) that, for $m>\frac{d}{p}$ and any $\delta_1,\delta_2\in\R$, 
pointwise multiplication of functions $(f,g)\mapsto fg$ defines a continuous map
\begin{equation*}
W^{m,p}_{\delta_1} \times W^{m,p}_{\delta_2} \to W^{m,p}_{\delta_1+\delta_2+\frac{d}{p}}.
\end{equation*}
This property is needed for the proof of the following Proposition (see, e.g., \cite[Appendix B]{McOwenTopalov2}).

\begin{Prop}\label{prop:A-products}
\begin{itemize}
\item[(i)]Suppose $m>\frac{d}{p}$,  $0\leq n_i\leq N_i$,  and $\ell_i+n_i\geq 0$  for $i=1,2$. Let $n_0:=n_1+n_2$ and
$\ell_0:=\ell_1+\ell_2$. Then, 
\begin{equation}
\|u\,v\|_{{\mathcal A}_{n_0,N_0;-n_0}^{m,p} }\le C\,\|u\|_{{\mathcal A}_{n_1,N_1;-n_1}^{m,p} }
\|v\|_{{\mathcal A}_{n_2,N_2;-n_2}^{m,p} }
\ \hbox{for}\ u\in {\mathcal A}_{n_1,N_1;-n_1}^{m,p},\ v\in {\mathcal A}_{n_2,N_2;-n_2}^{m,p}
\end{equation}
where $N_0:=\min(N_1+n_2,N_2+n_1)$.
\item[(ii)] If $m>\frac{d}{p}$ and $-n\le\ell\le 0$, then $\A^{m,p}_{n,N;\ell}$ is a Banach algebra. 
In particular, $\A^{m,p}_{n,N;0}$ is a Banach algebra if $m>\frac{d}{p}$ and $0\le n\le N$.
\end{itemize}
\end{Prop}

These statements hold without changes for the corresponding complexified spaces.
In what follows we also collect some of the basic properties of the group of asymptotic diffeomorphisms 
(see \cite[Appendix B]{McOwenTopalov2}, \cite[Section 2]{McOwenTopalov3}).

\begin{Th}\label{th:A-composition}
Assume that $m>1+\frac{d}{p}$, $0\le n\le N$, and $-n\le\ell\le 0$. Then we have:
\begin{itemize}
\item[(i)] The composition map 
$\A^{m,p}_{n,N;\ell}\times\A D^{m,p}_{N;\ell}\to\A^{m,p}_{n,N;\ell}$, $(u,\varphi)\mapsto u\circ\varphi$, 
and the inverse map
$\A D^{m+1,p}_{n,N;\ell}\to\A D^{m+1,p}_{n,N;\ell}$, $\varphi\mapsto\varphi^{-1}$, are continuous.
\item[(ii)] The maps $\A^{m+1,p}_{n,N;\ell}\to\A^{m,p}_{n,N;\ell}$, $\varphi\mapsto\varphi^{-1}$,  and
$\A^{m+1,p}_{n,N;\ell}\times\A D^{m,p}_{N;\ell}\to\A^{m,p}_{n,N;\ell}$, $(u,\varphi)\mapsto u\circ\varphi$,
are $C^1$.
\end{itemize}
In particular, $\A D^{m,p}_{N;\ell}$ is a topological group for $m>2+\frac{d}{p}$.
\end{Th}

The composition has the following stronger property:

\begin{Prop}\label{prop:A-composition}
Assume that $m>1+\frac{d}{p}$, $0\le n\le N$, and $-n\le\ell\le 0$. Then, the map
$\A^{m,p}_{n,N+n;-n}\times\A D^{m,p}_{N;0}\to\A^{m,p}_{n,N+n;-n}$, $(u,\varphi)\mapsto u\circ\varphi$, 
is continuous, and the map
$\A^{m+1,p}_{n,N+n;-n}\times\A D^{m,p}_{N;0}\to\A^{m,p}_{n,N+n;-n}$, $(u,\varphi)\mapsto u\circ\varphi$, 
is $C^1$.
\end{Prop}

\begin{proof}[Proof of Proposition \ref{prop:A-composition}]
The fact that the map
$\A^{m,p}_{n,N+n;-n}\times\A D^{m,p}_{N;0}\to\A^{m,p}_{n,N+n;-n}$, $(u,\varphi)\mapsto u\circ\varphi$, 
is well defined follows from \cite[Corollary 3.1]{McOwenTopalov3}. In view of \cite[Corollary 5.1]{McOwenTopalov2},
the other two properties follow easily by inspection along the lines of the proof of \cite[Lemma 3.1]{McOwenTopalov3}.
\end{proof}

\end{document}